\documentclass[12pt,a4paper,reqno]{amsart}
\usepackage{relsize,exscale}
\usepackage{hyperref}
\newcommand*{\affaddr}[1]{#1}
\newcommand*{\affmark}[1][*]{\textsuperscript{#1}}
\usepackage{graphics,layout,indentfirst,amsmath,amsthm,amssymb}
\usepackage{tgpagella}
\newtheorem{theorem}{Theorem}

\newtheorem{lemma}[theorem]{Lemma}

\newtheorem{definition}{Definition}

\numberwithin{equation}{section}
\makeatletter
\@namedef{subjclassname@2010}{2010 Mathematics Subject
Classification} \makeatother
\begin{document}
\title[ Convergence of Random Jacobi Series in weighted $\mathrm{L}_{[-1,1]}^{\mathrm{p},(\zeta,\eta)}$ Space ]
{ On the Convergence of Random Fourier--Jacobi Series in weighted $\mathrm{L}_{[-1,1]}^{\mathrm{p},(\zeta,\eta)}$ Space }

\author[P. Maharana, S. Sahoo ]
{P\lowercase{artiswari} M\lowercase{aharana}\affmark[1] \lowercase{and}
S\lowercase{abita} S\lowercase{ahoo}\affmark[2]\\
\affaddr{\affmark[1] D\lowercase{epartment} \lowercase{of} M\lowercase{athematics}, S\lowercase{ambalpur} U\lowercase{niversity}, O\lowercase{disha}, I\lowercase{ndia}}\\
\affaddr{\affmark[2] D\lowercase{epartment} \lowercase{of} M\lowercase{athematics}, S\lowercase{ambalpur} U\lowercase{niversity}, O\lowercase{disha}, I\lowercase{ndia}.
}\\
{\affmark[1] \lowercase{partiswarimath1@suniv.ac.in} }\\
{\affmark[2] \lowercase{sabitamath@suniv.ac.in} } \\
 }
\begin{abstract}
In the present paper, the random series $\sum\limits_{m=0}^\infty c_m C_m(\varpi)q_m^{(\zeta,\eta)}(u)$ in orthogonal Jacobi polynomials
$q_m^{(\zeta,\eta)}(u)$ is discussed.
The scalars $c_m$ are Fourier--Jacobi coefficients of a function in the weighted space $\mathrm{L}_{[-1,1]}^\mathrm{p}(d\mu_{\zeta,\eta}),\mathrm{p}>1.$
The random variables $C_m(\varpi)$
are chosen to be the Fourier--Jacobi coefficients of symmetric stable process $Y_{\zeta,\eta}(v,\varpi)$ of index $\chi \in [1,2]$ for $\zeta,\eta \geq 0,$ which are not independent.
We prove that, under certain conditions on $\mathrm{p},\zeta$ and $\eta,$ 
the random Fourier--Jacobi series converges in probability
to the stochastic integral 
\begin{equation*}
\int_{-1}^1 \mathfrak{g}(u,v)dY_{\zeta,\eta}(v,\varpi).
\end{equation*}
We also establish the existence of this integral in the sense of probability
 for $\mathfrak{g} \in \mathrm{L}_{[-1,1]}^\mathrm{p}(d\mu_{\zeta,\eta}).$  



\end{abstract}
\maketitle{}
\section{\bf Introduction}
\setcounter{equation}{0}

The field of Fourier analysis has been significantly enriched by the introduction of Fourier series and it has extensive applications in the physical sciences and engineering. 
Studies on Fourier series in orthogonal polynomials have received attention due to the need for the expression of functions as a series in
various orthogonal functions.
Many branches of the physical sciences use these expression of functions as a series in orthogonal polynomials.
In the early part of the year 1800, the Fourier series was developed to solve the problem of heat diffusion in a continuous medium. 
 This concept was later on expanded to random Fourier series in 1930 by the work of Kahane, Marcus, and Pisier. The random Fourier series have many applications in various fields, such as electrical engineering, signal processing, optics, etc.
Random Fourier series play a vital role in electrical engineering, signal processing, optics, and other fields because white noise, which is inherently present in all these areas, is a random signal.
 It is also used for the generation of random noise \cite{GA}. 
Most of the random series that have been studied are formed of independent random coefficients.
Rademacher \cite{HR}, Steinhaus \cite{HS1,HS2}, Hunt \cite{GH}, Paley, and Zygmund \cite{PZ1.1,PZ1.2}, Paley et al. \cite{PWZ} discussed random series
 in the form
 \begin{equation}\label{00.1}
\sum_{n=-\infty}^\infty a_n X_n(\varpi) e^{int}.
\end{equation}
The random coefficients $X_n(\varpi)$ are independent, identically distributed random variables, such as 
Rademacher functions, Steinhaus functions, etc.
 Hide et al. \cite{HKPS} investigated  
 the series (\ref{00.1}) in the context of the study of ``White Noise."
The model for white noise was taken to be $\displaystyle{{dW}\over{dt}},$ where $W(t)$ is the Brownian motion. 
It is also known as Wiener process, 
which is a continuous stochastic process with independent increments such that
\begin{equation*}
E\Big(W(t)-W(s)\Big)=0\; \text{and} \; E\Big|W(t)-W(s)\Big|^2=|t-s|^2.  
\end{equation*}
Since $W(t)$ is nowhere differentiable, this model seems meaningless and can be assigned in a weak sense. 
It is observed that the stable process $\mathfrak{X}(v,\varpi)$ are better models for white noise than the Wiener process. 
Stochastic process are useful in modeling extremal events. Recently, it has been found that the
stable process $\mathfrak{X}(v,\varpi)$ is applied to the analysis of systems under heavy--tailed noise, such as in assessing wind power fluctuations \cite{KAO2019}.
Many researchers, such as Nayak et al. \cite{NPM}, Dash and Pattanayak \cite{DP}, Dash et al. \cite{DPP} have studied the random Fourier series associated with 
the stable process $\mathfrak{X}(v,\varpi).$

Further, the orthogonal polynomials, which originated in the 19th century, play an essential role in mathematical physics. Extensive studies have been conducted on the 
Fourier series in orthogonal polynomials. Marian and Marian \cite{MM} have studied power series involving orthogonal polynomials that occur in problems of quantum optics. 
The Jacobi polynomials hold great importance in the hierarchy of orthogonal polynomial classes and have found widespread use in all areas of science and engineering. Many researcher have discussed the convergence property of Fourier–Jacobi series.
However, research on random Fourier series involving orthogonal polynomials has not been extensively explored. The complexity involved in handling such series and the limited availability of literature on their application in physical problems are among the primary reasons for this lack of attention.

Liu and Liu \cite{LL, LL1} discussed the Fourier series in Hermite polynomials for functions $f$ belongs to 
$\mathrm{L}^2(\mathbb{R})$ in the form
\begin{equation*}
\mathfrak{g}(u):=\sum\limits_{m=1}^\infty d_m \varphi_m(u),
\end{equation*}
where
\begin{equation*}
d_m:=\int_{-\infty}^\infty \mathfrak{g}(u)\varphi_m(u) du
\end{equation*}
are the Fourier--Hermite coefficients of $\mathfrak{g}$ and $\varphi_m(u)$
 are Hermite Gaussian functions.
Since the eigenfunctions of Fourier transform defined on 
$\mathrm{L}^2(\mathbb{R})$ 
 are the Hermite Gaussian functions with corresponding eigenvalues
\begin{equation*}
\lambda_m:=\exp\Big[\frac{-i \pi}{2}\mod(m,4)\Big],
\end{equation*}
that leads to express the Fourier transform $\mathfrak{G}$ of $\mathfrak{g}$ as
\begin{equation*}
\sum_{m=1}^\infty d_m \lambda_m \varphi_m(u).
\end{equation*}
Further, the fractional Fourier transform of $\mathfrak{g}$
of rational order $\beta$ have eigenfunctions same to the Hermite Gaussian functions, 
but correspond to the eigenvalues $\lambda_m^\beta.$
This
leads 
to the expression of fractional Fourier transform 
$\mathfrak{G}^\beta$ of $\mathfrak{g}$
 in series form as
\begin{equation*}
\sum_{m=0}^\infty d_m \lambda_m^\beta \varphi_m(u).
\end{equation*}
Liu and Liu extended the rational order fractional Fourier transform of $\mathfrak{g}$
to irrational order and introduced the series
\begin{equation}\label{0.5}
\mathcal{R}[\mathfrak{g}(u)]:=\sum_{m=0}^\infty d_m \mathcal{R}(\lambda_m) \varphi_m(u),
\end{equation}
where $\mathcal{R}(\lambda_m):=\exp[i\pi \mathrm{Random(m)}]$ are 
random values from the unit circle in $\mathbb{C}.$
They called it the random Fourier transform, which is a random Fourier series in orthogonal Hermite polynomials.
They applied it in image encryption and decryption as well as in general image processing and signal processing.
This raises the question of what would happen if the random coefficients $\mathcal{R}(\lambda_n)$ selected from the 
unit circle in $\mathbb{C}$ are replaced by other random variables in the random series (\ref{0.5}). 
All these works and the application of Jacobi polynomials in mathematical physics \cite{J},
motivated us and provide a way to study random Fourier series involving orthogonal Jacobi polynomials. 

The random series considered here is in the form
\begin{equation}\label{0.3}
\sum_{m=0}^\infty c_m r_m(\varpi) q_m^{(\zeta,\eta)}(u),
\end{equation}
where $q_m^{(\zeta,\eta)}(u)$ is the orthogonal Jacobi polynomials  
with a distinct set of random variables $r_m(\varpi).$
Recently, we studied the random Jacobi series (\ref{0.3}) for specific classes of continuous functions \cite{MS}. In this article, 
we discuss the convergence of random series in weighted measurable space $\mathrm{L}_{[-1,1]}^\mathrm{p}(d\mu_{\zeta,\eta}).$
\begin{definition}
The space $\mathrm{L}_{[-1,1]}^\mathrm{p}(d\mu_{\zeta,\eta}),\mathrm{p} \geq 1$ is the set of all measurable functions $\mathfrak{g}$
in the segment $[-1,1]$ with
the weighted measure 
\begin{equation*}
d\mu_{\zeta,\eta}(u):=\varrho^{(\zeta,\eta)}(u)du,  
\end{equation*}
where
\begin{equation*}
\varrho^{(\zeta, \eta)}(u):=(1-u)^\zeta(1+u)^\eta,\zeta,\eta>-1,  
\end{equation*}
and satisfy 
\begin{equation*}
\int\limits_{-1}^1 |\mathfrak{g}(u)|^\mathfrak{p} d\mu_{\zeta,\eta}(u) < \infty.
\end{equation*}
\end{definition}
This space is equipped with the weighted norm
\begin{equation*}
||\mathfrak{g}||_{\mathrm{L}_{[-1,1]}^\mathrm{p}(d\mu_{\zeta, \eta})}=\Big\{\int_{-1}^1 |\mathfrak{g}(u)|^\mathrm{p} d\mu_{\zeta, \eta}(u)\Big\}^{\frac{1}{\mathrm{p}}},\; \text{for} \; \mathrm{p} \geq 1.
\end{equation*}
If $\zeta =\eta=0,$ then
it is simply the $\mathrm{L}_{[-1,1]}^\mathrm{p}$ space and
\begin{equation*}
||\mathfrak{g}||_{\mathrm{L}_{[-1,1]}^\mathrm{p}(d\mu_{0,0})}=||f||_{\mathrm{p}}=\Big\{\int_{-1}^1|\mathfrak{g}(u)|^\mathrm{p}du\Big\}^{1/\mathrm{p}}.
\end{equation*}
The scalars $c_m$ are considered to be the Fourier--Jacobi coefficients of a function $\mathfrak{g} \in \mathrm{L}_{[-1,1]}^\mathrm{p}(d\mu_{\zeta,\eta}),\mathrm{p}\geq1$ defined as
\begin{equation}\label{1.2}
c_m:=\int_{-1}^1 \mathfrak{g}(v)q_m^{(\zeta,\eta)}(v)d\mu_{\zeta,\eta}(v).
\end{equation}

The random variables $r_m(\varpi)$
are chosen to be the Fourier--Jacobi coefficients of a symmetric stable process. According to Lukacs \cite[p.~148]{L}, the stochastic integral
\begin{equation}\label{1.1.0}
\int\limits_a^b \mathfrak{g}(v)d\mathfrak{X}(v,\varpi)
\end{equation}
exists in probability if $\mathfrak{X}(v,\varpi),\;v\in \mathbb{R}$ 
is a continuous stochastic process and $\mathfrak{g}$
is a continuous function on $[a,b].$
The integral (\ref{1.1.0}) is known as a random variable.
Moreover, if $\mathfrak{X}(v,\varpi),\;v \in \mathbb{R}$ 
is a symmetric stable process of index $\chi \in [1,2]$
and $\mathfrak{g} \in \mathrm{L}_{[a,b]}^\mathrm{p},$ 
$\mathrm{p}\geq 1,$ 
 then the integral (\ref{1.1.0}) exists in probability 
for $\mathrm{p} \geq \chi$ 
\cite{NPM}.
Consider the Jacobi weight
$\varrho^{(\zeta, \eta)}(v):={(1-v)}^\zeta{(1+v)}^\eta,$ which is continuous in $[-1,1],$ for $\zeta, \eta \geq 0.$
Therefore, the integral $\int_{-1}^v \varrho^{(\zeta,\eta)}(s)d\mathfrak{X}(s,\varpi)$ exists and is a random variable.
Let 
\begin{equation*}
Y_{\zeta,\eta}(v,\varpi):=\int_{-1}^{v}\varrho^{(\zeta,\eta)}(s)d\mathfrak{X}(s,\varpi),
\end{equation*}
which is a stochastic process with $v \in [-1,1].$
 It has been seen that $Y_{\zeta,\eta}(v,\varpi)$ is a symmetric stable process of index $\chi \in [1,2]$ (see Lemma \ref{4L.1}).

It is proved that if $\mathfrak{g}$  belongs to $\mathrm{L}_{[-1,1]}^\mathrm{p}{(d\mu_{\zeta,\eta})},\zeta,\eta \geq 1,$
then the integral
\begin{equation}\label{0.2}
\int\limits_{-1}^1 \mathfrak{g}(v)dY_{\zeta,\eta}(v,\varpi)
\end{equation}
 exists in probability in the segment $[-1,1]$ for $\mathrm{p}\geq \chi$ (see Theorem \ref{4T.1}).
If $\mathfrak{g}(v):=q_m^{(\zeta,\eta)}(v)$ is the orthonormal Jacobi polynomials, then
\begin{equation}\label{2.2}
C_m(\varpi):=\int\limits_{-1}^1 q_m^{(\zeta,\eta)}(v) dY_{\zeta,\eta}(v,\varpi)
\end{equation}
exist in $[-1,1],$ which are random variables.
We call these $C_m(\varpi)$ as the Fourier--Jacobi coefficients of the symmetric
stable process $Y_{\zeta,\eta}(v,\varpi),$ which can be seen easily to be not independent.
The convergence of random series (\ref{0.3}) is established under some restrictions on the parameters $\zeta,\eta.$

This paper is structured as follows.
Some auxiliary results that will be used to prove the convergence of the random series (\ref{0.3}) is established in Section 2.
In Section 3, we prove 
the existence of integral (\ref{0.2}) and show that the random Fourier--Jacobi series (\ref{0.3}) converges to the integral (\ref{0.2}).
\section{\bf Auxiliary Results}
\begin{lemma}\label{4L.1}
 Let $\mathfrak{X}(v,\varpi)$ be a symmetric stable process and $\varrho^{(\zeta,\eta)}(v)$ be the Jacobi weight in $[-1.1].$ Then $Y_{\zeta,\eta}(v,\varpi)$ is a symmetric stable process of index $\chi \in [0,2],$ for $\zeta,\eta \geq 0$ and $v \in [-1,1].$
\end{lemma}
\begin{proof}
Let $\mathfrak{X}^{(1)}(v,\varpi)$ and $\mathfrak{X}^{(2)}(v,\varpi)$ be independent copies of the random variable $\mathfrak{X}(v,\varpi).$ Consider the Jacobi weight
function $\varrho^{(\zeta,\eta)}(v,\varpi)$ in the interval $[-1,1],$ for $\zeta,\eta \geq 0.$
 Define
 \begin{equation*}
 Y_{\zeta,\eta}^{(1)}(v,\varpi):=\int_{-1}^v \varrho^{(\zeta,\eta)}(s)d\mathfrak{X}^{(1)}(s,\varpi)\; \text{and} \;Y_{\zeta,\eta}^{(2)}(v,\varpi):=\int_{-1}^v \varrho^{(\zeta,\eta)}(s)d\mathfrak{X}^{(2)}(s,\varpi).  
 \end{equation*} 
The $Y_{\zeta,\eta}^{(1)}(v,\varpi)$ and $Y_{\zeta,\eta}^{(2)}(v,\varpi)$ are random copies of $Y_{\zeta,\eta}(v,\varpi).$
Now fix $A>0$ and $B>0.$ Since $\mathfrak{X}(v,\varpi)$ is stable, there exist a $C>0$ and $D \in \mathbb{R}$ such that $A\mathfrak{X}^{(1)}(v,\varpi)+B\mathfrak{X}^{(2)}(v,\varpi)$
is equality in distribution to $C\mathfrak{X}(v,\varpi)+D$ \cite[Definition 2.1.1]{ST}.
In the case of symmetric stable $D=0.$
Now
\begin{equation*}
A\mathfrak{X}^{(1)}(v,\varpi)+B\mathfrak{X}^{(2)}(v,\varpi) = C\mathfrak{X}(v,\varpi),
\end{equation*}
i.e.,
\begin{eqnarray*}
 A\int_{-1}^{v}\varrho^{(\zeta,\eta)}(s) d\mathfrak{X}^{(1)}(s,\varpi)&+&B \int_{-1}^{v} \varrho^{(\zeta,\eta)}(s) d\mathfrak{X}^{(2)}(s,\varpi)\\
 &&=C\int_{-1}^{v}\varrho^{(\zeta,\eta)}(s)d\mathfrak{X}(s,\varpi),  
\end{eqnarray*}
i.e.,
\begin{equation}\label{d.2}
 AY_{\zeta,\eta}^{(1)}(v,\varpi)+BY_{\zeta,\eta}^{(2)}(v,\varpi)= CY_{\zeta,\eta}(v,\varpi).
\end{equation}
The linear combination in left hand side of (\ref{d.2}) have same distribution as that of the right hand side.
This implies that $Y_{\zeta,\eta}(v,\varpi)$ is a symmetric stable process.
\end{proof}
To establish the existence of the integral
\begin{equation}\label{5.1}
\int_{-1}^1 \mathfrak{g}(v)dY_{\zeta,\eta}(v,\varpi)
\end{equation}
 in probability, 
we will proceed through establishing it for functions in a dense subset $\mathrm{C}_{[-1,1]}^{(\zeta,\eta)}$ of the space
$ \mathrm{L}_{[-1,1]}^p(d\mu_{\zeta,\eta}).$
\begin{definition}
The space $\mathrm{C}_{[-1,1]}^{(\zeta,\eta)}$ is the set of all real valued continuous functions on $(-1,1)$ for which
\begin{equation*}
\lim_{|v| \rightarrow 1} \mathfrak{g}(v) \varrho^{(\zeta,\eta)}(v)=0,
\end{equation*}
where $\varrho^{(\zeta,\eta)}(v):={(1-v)}^\zeta {(1+v)}^\eta.$
\end{definition}
The norm endowed in this space is
\begin{equation*}
||\mathfrak{g}||_{\mathrm{C}_\infty(d\mu_{\zeta,\eta})}=||\mathfrak{g}\varrho^{(\zeta,\eta)}(v)||:=\max_{v \in [-1,1]}|\mathfrak{g}(v)|\varrho^{(\zeta,\eta)}(v).
\end{equation*}
The space $\mathrm{C}_{[-1,1]}^{(\zeta,\eta)}$ is complete with respect to this norm. It is dense in the space
$ \mathrm{L}_{[-1,1]}^p(d\mu_{\zeta,\eta}).$
\begin{definition}\label{d.5.1}
The space $\mathrm{C}_{[-1,1]}^{'(\zeta,\eta)}$ is the set of all real valued continuous functions in $\mathrm{C}_{[-1,1]}^{(\zeta,\eta)}$ with continuous derivative.  
\end{definition}

The space $\mathrm{C}_{[-1,1]}^{'(\zeta,\eta)}$ is dense in $\mathrm{C}_{[-1,1]}^{(\zeta,\eta)},$ and hence dense in $\mathrm{L}_{[-1,1]}^p(d\mu_{\zeta,\eta}).$ 
It is known that if $X(t,\omega)$ is a continuous stochastic process and $f$ is a continuous function in $[a,b],$ then the stochastic integral
\begin{equation}\label{4.0.1}
\int_a^b \mathfrak{g}(v)d\mathfrak{X}(v,\varpi)
\end{equation}
exists in probability (c.f. Lukacs \cite{L}).
Followed by this,
the stochastic integral \eqref{5.1} exists in probability in the interval $[-1,1]$ 
 for $\mathfrak{g} \in \mathrm{C}_{[-1,1]}^{'(\zeta,\eta)},$ since $\varrho^{(\zeta,\eta)}(t)$ is continuous in $[-1,1]$ for $\zeta,\eta \geq 0.$

\section{Main Results}
\subsection{Existence of the stochastic integral}
Let $\mathfrak{X}(v,\varpi)$ be a stable process of index $\chi \in [1,2]$ which is symmetric and consider the Jacobi weight $\varrho^{(\zeta,\eta)}(v):={(1-v)}^\zeta{(1+v)}^\eta$ in the interval $[-1,1].$ Then
\begin{equation}\label{0.4}
 Y_{\zeta,\eta}(v,\varpi):=\int_{-1}^{v}\varrho^{(\zeta,\eta)}(s)d\mathfrak{X}(s,\varpi)
\end{equation}
is a symmetric stable process of index $\chi \in [1,2]$ for $v \in [-1,1].$
\\
The inequality which will be used to prove the Theorem \ref{4T.1} is established in the following lemma.
\begin{lemma} \label{4L.4}
Let $\mathfrak{g}$ be any function in $\mathrm{C}_{[-1,1]}^{'(\zeta,\eta)}.$ If $Y_{\zeta,\eta}(v,\varpi)$ is a
 symmetric stable process of index $\chi \in (0,2],$ then for all $\varepsilon >0$ and $\zeta,\eta \geq 0,$
\begin{equation*}
P\Bigg(\Bigg|\int_{-1}^1 \mathfrak{g}(v)dY_{\zeta,\eta}(v,\varpi)\Bigg|> \epsilon\Bigg)
\leq
\frac{2^{\chi+1}C}{(\chi+1)\varepsilon^\chi}\int_{-1}^1 |\mathfrak{g}(v)|^\chi d\mu_{\zeta,\eta}(v),
\end{equation*}
where $C$ is a positive constant.
\end{lemma}
\begin{proof}
Let $\mathfrak{g}$ be any function in $\mathrm{C}_{[-1,1]}^{'(\zeta,\eta)}$ and $Y_{\zeta,\eta}(v,\varpi)$ be a
 symmetric stable process of index $\chi \in (0,2].$
Now by Lemma 3.12 of \cite{SM},
\begin{eqnarray*}
P\Bigg(\Bigg|\int_{-1}^1 \mathfrak{g}(v)dY_{\zeta,\eta}(v,\varpi)\Bigg|> \epsilon\Bigg)
&=&P\Bigg(\Bigg|\int_{-1}^1 \mathfrak{g}(v)\varrho^{(\zeta,\eta)}(v) d\mathfrak{X}(v,\varpi)\Bigg|> \epsilon\Bigg)\\
&\leq&
\frac{2^{\chi+1}C}{(\chi+1)\varepsilon^\chi}\int_{-1}^1 |\mathfrak{g}(v)|^\chi d\mu_{\zeta,\eta}(v)
\end{eqnarray*}
Since $\varrho^{(\zeta,\eta)}(v)$ is bounded in the interval $[-1,1],$ thus the last expression in the above lemma is bounded by $\int\limits_{-1}^1 |\mathfrak{g}(v)|^\chi d\mu_{\zeta,\eta}(v).$ 
\end{proof}
\begin{theorem}\label{4T.1}
If $\mathfrak{g} \in \mathrm{L}_{[-1,1]}^\mathrm{p}(d\mu_{\zeta,\eta}),$ then the integral (\ref{5.1})
can be defined in probability for $\mathrm{p} \geq \chi \geq 1$ and $\zeta, \eta \geq 0.$
\end{theorem}
\begin{proof}
 If $\mathfrak{g} \in \mathrm{L}_{[-1,1]}^\mathrm{p}(d\mu_{\gamma,\delta}),$ where $d\mu_{\zeta,\eta}(v)=\varrho^{(\zeta,\eta)}(v)dv,$ then there exist a sequence of continuous functions $\mathfrak{g}_m$ in the interval $ [-1,1],$ 
 such that
\begin{equation}\label{4.2.1}
\lim_{m \rightarrow \infty} \int\limits_{-1}^1 |\mathfrak{g}_m(v)-\mathfrak{g}(v)|^\mathrm{p} d\mu_{\zeta,\eta}(v)=0.
\end{equation}
Thus,
\begin{equation}\label{4.2.2}
\lim_{n,m \rightarrow \infty}\int\limits_{-1}^1 |\mathfrak{g}_m(v)-\mathfrak{g}_n(v)|^\mathrm{p} d\mu_{\zeta,\eta}(v)=0.
\end{equation}
By Lemma \ref{4L.4}, for $\mathrm{p} \geq \chi \in [1,2]$ and $\zeta,\eta \geq 0,$
\begin{equation*}
\begin{split}
&P\Bigg(\Bigg|\int\limits_{-1}^1 \mathfrak{g}_n(v)dY_{\zeta,\eta}(v,\varpi)- \int\limits_{-1}^1 \mathfrak{g}_m(v)dY_{\zeta,\eta}(v,\varpi)\Bigg|> \varepsilon\Bigg)\\
&=P\Bigg(\Bigg|\int\limits_{-1}^1 \Big(\mathfrak{g}_n(v)- \mathfrak{g}_m(v)\Big)dY_{\zeta,\eta}(v,\varpi)\Bigg|> \varepsilon\Bigg)\\
&\leq
\frac{2^{\chi+1}C}{(\chi+1)\varepsilon^\chi}\Big\{\int_{-1}^1 |\mathfrak{g}_n(v)-\mathfrak{g}_m(v)|^\chi d\mu_{\zeta,\eta}(v)\Big\}.
\end{split}
\end{equation*}
By equation (\ref{4.2.2}),
we get
\begin{equation*}
\lim_{n,m \rightarrow \infty} P\Bigg(\Bigg|\int\limits_{-1}^1 \mathfrak{g}_n(v)dY_{\zeta,\eta}(v,\varpi)- \int\limits_{-1}^1 \mathfrak{g}_m(v)dY_{\zeta,\eta}(v,\varpi)\Bigg|> \varepsilon\Bigg)=0.
\end{equation*}
Thus the integral 
\begin{equation*}
\int_{-1}^1 \mathfrak{g}_m(v) dY_{\zeta,\eta}(v,\varpi)
\end{equation*}
 converges in probability, for $\zeta,\eta \geq 0.$
 Hence by Kawata \cite[p.~479]{K1}, there exists a random variable $\mathfrak{X}$ such that 
 \begin{equation*}
\lim_{m\rightarrow \infty}P\Bigg(\Bigg|\int\limits_{-1}^1 \mathfrak{g}_m(v)dY_{\zeta,\eta}(v,\varpi)-\mathfrak{X}\Bigg|> \varepsilon\Bigg)=0.
 \end{equation*}
This is independent of choice of sequence of functions $\mathfrak{g}_m(v).$
Suppose $\mathfrak{g} \in \mathrm{L}_{[-1,1]}^\mathrm{p}(d\mu_{\zeta,\eta}),$ we can obtain a continuous sequence $h_m(v)$ in the segment $ [-1,1],$ such that
\begin{equation*}
\lim_{m \rightarrow \infty}\int\limits_{-1}^1 \Big|h_m(v)-\mathfrak{g}(v) \Big|^\mathrm{p} d\mu_{\zeta,\eta}(v)=0.
\end{equation*}
Thus
\begin{equation}\label{4.2.3}
\lim_{m \rightarrow \infty}\int_{-1}^1 |\mathfrak{g}_m(v)-h_m(v)|^\mathrm{p} d\mu_{\zeta,\eta}(v)=0.
\end{equation}
By Lemma \ref{4L.4} and equation (\ref{4.2.3}), we get
\begin{equation*}
\lim_{m \rightarrow \infty}P\Bigg(\Bigg|\int\limits_{-1}^1 \mathfrak{g}_m(v)dY_{\zeta,\eta}(v,\varpi)-\int\limits_{-1}^1 h_m(v) dY_{\zeta,\eta}(v,\varpi)\Bigg|> \varepsilon\Bigg)=0.
\end{equation*}
Hence the integral $\int\limits_{-1}^1 \mathfrak{g}_m(v)dY_{\zeta,\eta}(v,\varpi)$ converges uniquely in probability to $\int\limits_{-1}^1 \mathfrak{g}(v)dY_{\zeta,\eta}(v,\varpi).$
\end{proof}

 The following lemma
 is an extension of Lemma \ref{4L.4}, for the functions $\mathfrak{g}$ in $\mathrm{L}_{[-1,1]}^\mathrm{p}(d\mu_{\zeta,\eta}).$
\begin{lemma}\label{4L.7}
Let $\mathfrak{g} \in \mathrm{L}_{[-1,1]}^\mathrm{p}(d\mu_{\zeta,\eta}),\; \mathrm{p} \geq \chi,$ 
then
\begin{equation*}
P\Bigg(\Bigg|\int_{-1}^1 \mathfrak{g}(v) dY_{\zeta,\eta}(v,\varpi)\Bigg|> \varepsilon\Bigg)
\leq
\frac{2^{\chi+1}C}{(\chi+1)\varepsilon'^\chi}\Big\{\int_{-1}^1 |\mathfrak{g}(v)|^\chi d\mu_{\zeta,\eta}(v)\Big\},  
\end{equation*}
where $C$ is a positive constant, for $\zeta,\eta \geq 0$ and $0 <\varepsilon'<\varepsilon.$
\end{lemma}
\begin{proof}
For $\mathfrak{g} \in \mathrm{L}_{[-1,1]}^\mathrm{p}(d\mu_{\zeta,\eta}),$ we can find a sequence of continuous function, such that the equation (\ref{4.2.1}) is holds.\\
Now by Lemma \ref{4L.4}, for $\varepsilon >0$ and $\zeta,\eta \geq 0,$ 
\begin{eqnarray*}
P\Bigg(\Bigg|\int_{-1}^1 \mathfrak{g}(v)dY_{\zeta,\eta}(v,\varpi)\Bigg|> \varepsilon\Bigg)
&=&P\Bigg(\Bigg|\int_{-1}^1 \Big(\mathfrak{g}(v)-\mathfrak{g}_m(v)\Big)dY_{\zeta,\eta}(v,\varpi)\Bigg|> \varepsilon\Bigg)\\
&\leq& P\Bigg(\Bigg|\int_{-1}^1 \Big(\mathfrak{g}(v)-\mathfrak{g}_m(v)\Big)dY_{\zeta,\eta}(v,\varpi)\Bigg|> \psi \Bigg)\\
&&+ P\Bigg(\Bigg|\int_{-1}^1 \mathfrak{g}_m(v)dY_{\zeta,\eta}(v,\varpi)\Bigg|> \varepsilon-\psi\Bigg).
\end{eqnarray*}
By Theorem \ref{4T.1}, 
\begin{equation*}
\lim_{m \rightarrow \infty}P\Bigg(\Bigg|\int_{-1}^1 \Big(\mathfrak{g}(v)-\mathfrak{g}_m(v)\Big)dY_{\zeta,\eta}(v,\varpi)\Bigg|> \psi \Bigg)=0.
\end{equation*}
With the help of Lemma \ref{4L.4} and Lebesgue dominated convergence theorem,
\begin{eqnarray*}
&&\lim_{m \rightarrow \infty} P\Bigg(\Bigg|\int_{-1}^1 \mathfrak{\mathfrak{g}}_m(v)dY_{\zeta,\eta}(v,\varpi)\Bigg|> \varepsilon-\psi\Bigg)\\
&&\leq \frac{2^{\chi+1}C}{(\chi+1){(\varepsilon-\psi)}^\chi}\lim_{m \rightarrow \infty}\Big\{\int_{-1}^1 |\mathfrak{g}_m(v)|^\chi d\mu_{\zeta,\eta}(v)\Big\}\\
&&= \frac{2^{\chi+1}C}{(\chi+1)\varepsilon'^\chi}\Big\{\int_{-1}^1 |\mathfrak{g}(v)|^\chi d\mu_{\zeta,\eta}(v)\Big\}.
\end{eqnarray*}
\end{proof}
\subsection{Convergence of random Fourier--Jacobi series:}
Consider the random Fourier--Jacobi series
\begin{equation}\label{1.11.1}
\sum_{m=0}^\infty c_m C_m q_m^{(\zeta,\eta)}(u)
\end{equation}
in orthonormal Jacobi polynomials $q_m^{(\zeta,\eta)}(u)$ with random coefficients $C_m(\varpi).$
The $C_m(\varpi)$ are Fourier--Jacobi coefficients of symmetric stable process
$Y_{\zeta,\eta}(v,\varpi)$ defined as in (\ref{2.2}).
The following theorem establishes the convergence of random series \eqref{1.11.1} in probability.
\begin{theorem}\label{4T.2}
 If $\mathfrak{g} \in \mathrm{L}_{[-1,1]}^\mathrm{p}(d\mu_{\zeta,\eta}),$
  and
$\mathrm{p}$ satisfies the condition
\begin{equation}\label{4.3.4}
4\;\max\Big\{\frac{\zeta+1}{2\zeta+3},\frac{\eta+1}{2\eta+3}\Big\}< p<
4\;\min\Big\{\frac{\zeta+1}{2\zeta+1},\frac{\eta+1}{2\eta+1}\Big\},
\end{equation}
then the random Fourier--Jacobi series (\ref{1.11.1})
converges in the sense of probability to
\begin{equation}\label{4.3.3}
\int_{-1}^1 \mathfrak{g}(u,v)dY_{\zeta,\eta}(v,\varpi),
\end{equation}
 for $\mathrm{p} \geq \chi >1$ and $\zeta,\eta \geq 0.$
\end{theorem}
\begin{proof}
 The $m$term approximation of the series (\ref{1.11.1}) is
\begin{equation*}
\mathbf{S}_m^{(\zeta,\eta)}(u,\varpi):=\sum_{j=0}^m c_j C_j(\varpi)q_j^{(\zeta,\eta)}(u).
\end{equation*}
The integral form of $\mathbf{S}_m^{(\zeta,\eta)}(u,\varpi)$ is
\begin{equation*}
\begin{split}
\mathbf{S}_m^{(\zeta,\eta)}(u,\varpi)&=\sum_{j=0}^m c_j C_j(\varpi)q_j^{(\zeta,\eta)}(u)\\
&= \int_{-1}^1  \mathbf{s}_m^{(\zeta,\eta)}(u,v) dY_{\zeta,\eta}(v,\varpi),
\end{split}
\end{equation*}
where
$\mathbf{s}_m^{(\zeta,\eta)}(u,v):=\sum\limits_{j=0}^m c_j q_j^{(\zeta,\eta)}(u)q_j^{(\zeta,\eta)}(v).$

With the help of Lemma \ref{4L.7}, for $\varepsilon'<\varepsilon$ and $\zeta,\eta \geq 0,$
\begin{equation*}
\begin{split}
&P\Bigg(\Bigg|\int_{-1}^1 \mathfrak{g}(u,v)dY_{\zeta,\eta}(v,\varpi)-\mathbf{S}_m^{(\zeta,\eta)}(v,\varpi)\Bigg|> \varepsilon \Bigg)\\
&= P\Bigg(\Bigg|\int_{-1}^1 \mathfrak{g}(u,v)dY_{\zeta,\eta}(v,\varpi)-\int_a^b \mathbf{s}_m^{(\zeta,\eta)}(u,v)dY_{\zeta,\eta}(v,\varpi)\Bigg|\Bigg)\\
&\leq \frac{2^{\chi+1}C}{\varepsilon'^\chi(\chi+1)}
\int_{-1}^1 \Big|\Big(\mathfrak{g}(u,v)-\mathbf{s}_m^{(\zeta,\eta)}(u,v)\Big)\Big|^\chi d\mu_{\zeta,\eta}(v).
\end{split}
\end{equation*}
It is known that,
for $\mathfrak{g} \in \mathrm{L}_{[-1,1]}^{\mathrm{p}}(d\mu_{\zeta,\eta})$ and $\zeta,\eta \geq 0,$ if $\mathrm{p}$
satisfy the condition (\ref{4.3.4}),
then by \cite[Theorem~A]{P2}
\begin{equation*}
\lim_{m \rightarrow \infty}\int_{-1}^1 \Big|\Big(\mathfrak{g}(v)-\mathbf{s}_m^{(\zeta,\eta)}(v)\Big)\Big|^\mathrm{p} d\mu_{\zeta,\eta}(v)=0.
\end{equation*}
Thus, for $\mathrm{p}\geq \chi,$
\begin{equation*}
\int_{-1}^1 |\mathfrak{g}(u,v)-\mathbf{s}_m^{(\zeta,\eta)}(u,v)|^\chi d\mu_{\zeta,\eta}(v)=0,
\; \text{as} \; m \rightarrow \infty,
\end{equation*}
which completes the proof.
\end{proof}
\section*{Acknowledgments}
I am thankful to University Grant Commission (National Fellowship with letter no-F./2015-16/NFO-2015-17-OBC-ORI-33062) for supporting this research work.

\end{document}